\documentclass[]{amsart}
\usepackage[utf8]{inputenc}
\usepackage{graphicx} 
\usepackage{amsfonts}
\usepackage{amsmath}
\usepackage{amsthm}
\usepackage{amssymb}
\usepackage{comment}
\usepackage{yhmath}
\usepackage{hyperref}

\newtheorem{thm}{Theorem}
\newtheorem{prop}[thm]{Proposition}
\newtheorem{lem}[thm]{Lemma}
\newtheorem{cor}[thm]{Corollary}

\newtheorem{rmk}[thm]{Remark}

\newtheorem{ex}[thm]{Example}
\newtheorem{defn}[thm]{Definition}

\numberwithin{thm}{section}
\numberwithin{equation}{section}

\newcommand{\Ric}{\mathrm{Ric}}

\usepackage{scalerel,stackengine}
\stackMath
\newcommand\widecheck[1]{%
\savestack{\tmpbox}{\stretchto{%
  \scaleto{%
    \scalerel*[\widthof{\ensuremath{#1}}]{\kern-.6pt\bigwedge\kern-.6pt}%
    {\rule[-\textheight/2]{1ex}{\textheight}}
  }{\textheight}%
}{0.5ex}}%
\stackon[1pt]{#1}{\scalebox{-1}{\tmpbox}}%
}

\title{On compact quasi-Einstein metrics of constant scalar curvature}
\author{Eric Cochran}
\date{}

\begin{document}

\maketitle

\begin{abstract}
    We show that all compact quasi-Einstein metrics of constant scalar curvature in dimension three are locally homogeneous.  We accomplish this by using the equivalence of constant scalar curvature quasi-Einstein metrics $(M,g,X)$ and quasi-Einstein metrics when $X$ is Killing to make a connection to Sasakian geometry in dimension three.  In higher dimensions, there are examples which are non-locally homogeneous with constant scalar curvature.  Such examples were constructed by Kunduri-Lucietti in \cite{Kunduri12} as circle bundles over a compact K\"ahler-Einstein base.  We then ask when compact quasi-Einstein metrics of constant scalar curvature can be constructed as circle bundles over Einstein metrics, and prove that the base must in fact be K\"ahler-Einstein, assuming a conjecture due to Goldberg (see \cite{Goldberg69}).  These spaces, in fact, admit one parameter families of quasi-Einstein metrics by considering the \textit{canonical variation}, which we study further.
\end{abstract}

\section{Introduction}

The notion of a \textit{quasi-Einstein metric} generalizes Einstein metrics by equipping a Riemannian manifold $(M,g)$ with a smooth vector field $X$. 

\begin{defn}
    Let $(M,g)$ be a Riemannian manifold, and let $X$ be a smooth vector field defined on $X$.  The triple $(M,g,X)$ is said to be \textbf{quasi-Einstein} if
        \begin{align} \label{qe}
            \Ric + \frac{1}{2}\mathcal{L}_Xg - \frac{1}{m}X^* \otimes X^* = \lambda g
        \end{align}
    where $m \neq 0$ and $\lambda$ are constants, $X^*$ is one form dual to the vector field $X$ induced by $g$, and $\mathcal{L}$ is the Lie derivative.
\end{defn}

\noindent  In this paper, we will assume $M$ has no boundary.  Observe that the above definition reduces to the Einstein condition when $X$ is identically zero (we will call such solutions \textit{trivial}).

The tensor on the left-hand side of (\ref{qe}) is known as the \textit{Bakry-\'Emery} Ricci tensor.  Bakry and \'Emery studied this tensor as it pertains to diffusion processes in \cite{Bakry85}.  Qian studied Comparison geometry for this tensor in \cite{Qian97}.

When $m=2$, solutions to (\ref{qe}) are known as \textit{near-horizon geometries}, and are therefore of interest in general relativity.  See \cite{Kunduri09} for a discussion.  In the limiting case, when $m=\infty$, solutions to (\ref{qe}) as known as \textit{Ricci solitons}, which are self-similar solutions to the Ricci flow.  

Case, Shu, and Wei studied solutions to (\ref{qe}) in the gradient case (i.e. where $X = \nabla \phi$ for some smooth function $\phi$ defined on $M$) in \cite{Case08}.  Among other things, they showed that there do not exist any non-trivial, compact quasi-Einstein metrics of constant scalar curvature.  It is natural to ask whether or not we can obtain solutions if we no longer require $X$ to be a gradient field.  If the $X$ is gradient condition is relaxed to the requirement that $dX^*=0$, work done in \cite{Bahuaud22} and \cite{Chrusciel05} shows that there are still no non-trivial solutions with constant scalar curvature when $\lambda \geq 0$, and when $\lambda < 0$. There is only a simple class of non-trivial examples in this case, namely those where $M = S^1 \times N$ is a product, and $X$ is constant length and tangent to the fibers of $S^1$ (see \cite{Wylie23}).  

The story is more interesting if we no longer assume $dX^*=0$.  In \cite{Chen16}, Chen-Liang-Zhu showed that most compact simple Lie groups admit non-trivial quasi-Einstein metrics.  In \cite{Lim22}, Lim studied and classified quasi-Einstein metrics on locally homogeneous 3-manifolds.  We recall the definition of locally homogeneous now:  

\begin{defn}
    A Riemannian manifold $(M,g)$ is called \textbf{locally homogeneous} if around any points $x, y \in M$, there exist neighborhoods $U_x$ of $x$ and $U_y$ of $y$ such that there exists an isometry $\Phi$ mapping $U_x$ onto $U_y$ such that $\Phi(x) = y$.
\end{defn}

Since locally homogeneous manifolds always have constant scalar curvature, Lim's examples show that removing the $dX^* = 0$ assumption allows for a number of more interesting examples.  Among these are circle bundles over $S^2$ known as \textit{Berger spheres}.

In the quest to find additional examples of compact quasi-Einstein metrics of constant scalar curvature, it is natural to ask whether or not there exist examples of quasi-Einstein metrics which have constant scalar curvature, but are not locally locally homogeneous.  This would provide a new class of examples in addition to the ones studied by Chen-Liang-Zhu and Lim mentioned above.  In higher dimensions, the answer to this question is ``yes".  A class of examples of such metrics can be realized as circle bundles over a K\"ahler-Einstein base, which are studied by Kunduri and Lucietti in \cite{Kunduri12}.

However, one of the main results in this paper is that the answer to the above question is ``no" in the three dimensional case. 

\begin{thm} \label{thm:1}
    Let $(M, g, X)$ be a quasi Einstein metric where $\mathrm{dim} \, M = 3$, $M$ is compact, and $(M, g)$ has constant scalar curvature.  Then $(M, g)$ is locally homogeneous.
\end{thm}

To obtain the above result, we use the fact that a compact quasi-Einstein metric of constant scalar curvature must have $X$ Killing. This result was first shown by Ghosh in Theorem 4.2 of \cite{Ghosh20}.  In fact, it follows from equation (5.18) of \cite{Ghosh20} that the converse is also true, and so the conditions of $(M,g)$ having constant scalar curvature and $X$ being Killing are equivalent in the compact case (see also \cite{Cochran2024} and \cite{CostaFilho24}, whose authors were unaware of Ghosh's results at the time of publication).  As shown in \cite{Bahuaud24}, it also follows that $X$ has constant length and that the integral curves of $X$ are geodesics, in addition to being Killing.  Thus, a compact quasi-Einstein metric $(M,g,X)$ of constant scalar curvature necessarily admits a constant length Killing field (namely $X$).  In Section 2, we outline how, in the three dimensional case, this shows that there exists a Sasakian structure of constant $\phi$-sectional curvature.  We then invoke a classification due to Tanno in \cite{Tanno69} to conclude that Theorem \ref{thm:1} holds.

Since there exists a constant length Killing field in the constant scalar curvature case, $M$ admits either an $S^1$ action or an $\mathbb{R}$ action by isometries generated by the flow of $X$ (depending on whether or not the integral curves generated by $X$ are closed or open, respectively).  If we assume that this action is a free $S^1$ action, there exists a Riemannian submersion $\pi: M \to B$ with totally geodesic $S^1$ fibers (the fact that the integral curves of $X$ are geodesics in this case in also covered in the result stated in \cite{Bahuaud24}).  We discuss this in more detail in Section 5, in particular, how it relates to the three dimensional case.  

A result due to \cite{Vilms70} (which is also stated in \cite{Besse87}) connects Riemannian submersion with totally geodesic fibers with principal $G$-bundles.  In particular, when $G=S^1$, the result connects principal circle bundles to Riemannian submersion with totally geodesic $S^1$ fibers.   

In Section 3, motivated by the examples in \cite{Kunduri12} and the above discussion, we find necessary conditions for when quasi-Einstein metrics of constant scalar curvature can be constructed as principal circle bundles over a compact base $B$.  Let $(M,g,X)$ be a quasi-Einstein metric of constant scalar curvature that admits a Riemannian submersion $\pi: (M,g) \to B$ with totally geodesic $S^1$ fibers such that $X$ is tangent to the fibers of $\pi$.  The first key observation we make is that $B$ must have constant scalar curvature.  This generalizes a result in \cite{Besse87}, which states that Einstein metrics constructed in this fashion must have $B$ with constant scalar curvature.  From there, we show that if we have an example of a quasi-Einstein metric of constant scalar curvature constructed as principal circle bundles of a compact Einstein base, then the base must have an almost K\"ahler structure.  A conjecture due to Goldberg in \cite{Goldberg69} postulates that the existence of an almost K\"ahler structure on a compact Einstein manifold $M$ implies $M$ is K\"ahler-Einstein.  In \cite{Sekigawa87}, Sekigawa shows this conjecture is true when $\lambda \geq 0$.  If this conjecture is true for all $\lambda$, then this would show that the examples over a K\"ahler-Einstein base are the only examples over a compact Einstein base.  

We state the result in terms of principal $S^1$ bundles, recalling that principal circle bundles over $B$ are characterized by elements of $H^2(B, \mathbb{Z})$.

\begin{thm} \label{thm:2}
    Let $(B, \check{g})$ be a compact Einstein manifold and let $p: P \to B$ be a principle $S^1$-bundle on $B$ with fibers of length $2\pi$, classified by $\alpha \in H^2(B, \mathbb{Z})$.  Suppose further that $X$ is a smooth vector field on $P$ tangent to the fibers of $\pi$.  If $(P,g,X)$ is an $S^1$-invariant quasi-Einstein metric of constant scalar curvature, then either:
        \begin{itemize}
            \item[(a)] $\check{\lambda} + \frac{|X|^2}{m} = 0$, $\alpha = 0$, and $P$ is the Riemannian product $B \times S^1$ (i.e. the circle bundle is trivial), or
            \item[(b)] $\check{\lambda} + \frac{|X|^2}{m} > 0$ and there exists an almost K\"ahler structure $(J, \check{g}, \omega')$ on $B$ such that $[\omega] \in H^2(B,\mathbb{Z})$ is a multiple of $\alpha$.
        \end{itemize}
    where $g$ is the unique metric on $P$ making $p$ a Riemannian submersion with totally geodesic fibers of length $2\pi$.
\end{thm}

In \cite{Besse87}, there is an analogous result stated in the Einstein case.  Something interesting to observe is that in the Einstein case, one must have $\check{\lambda} \geq 0$.  However, in the quasi-Einstein case, there are examples with $\check{\lambda} \leq 0$.

Furthermore, since almost K\"ahler structures can only exist on even dimensional manifolds, we easily obtain the following corollary.  This provides a nice rigidity result for when quasi-Einstein metrics of constant scalar curvature can be constructed over an Einstein base.

\begin{cor} \label{cor:1}
    Let $(B, \check{g})$ be a compact Einstein manifold, and suppose $p: P \to B$ is a non-trivial principle $S^1$-bundle over $B$ with $X$ tangent to the fibers of $p$.  If $P$ admits a quasi-Einstein metric of constant scalar curvature, then $P$ must be odd dimensional.
\end{cor}

Quasi-Einstein metrics of constant scalar curvature which are circle bundles over an odd dimensional base are possible.  In Section 3, we will elaborate on how an example in \cite{Valiyakath25} fits this criterion.  The base is not Einstein in this example, but it is locally homogeneous.

In Section 4, we study special kinds of deformations of these quasi-Einstein metrics with are constructed as principal $S^1$ bundles.  These deformations involves scaling the metric in the direction of the $S^1$ fibers by a factor of $t$.  This family of metrics is called the \textit{canonical variation}.  We show that two distinct quasi-Einstein metrics exist in the canonical variation if and only if the base space is Einstein.  Therefore, if $B$ is not Einstein, then there exists at most one quasi-Einstein metric in the canonical variation of $\pi$.  From this result, we obtain the following corollary.

\begin{cor} \label{cor:2}
    Let $(B, \check{g})$ be a compact Einstein manifold, and suppose $p: P \to B$ is a nontrivial principle $S^1$ bundle with $\check{\lambda} > 0$ and $m>0$.  Then $P$ admits an Einstein metric if and only if $P$ admits a quasi-Einstein metric.  Moreover, $P$ admits a 1-parameter family of quasi-Einstein metics in this case, including metrics with $\lambda > 0$, $\lambda=0$, and $\lambda<0$. 
\end{cor}

It is interesting to compare this to the Einstein case.  Theorem 9.73 of \cite{Besse87} states in part that, if there are two Einstein metrics in the canonical variation of $\pi$, then the fibers of $\pi$ must have positive scalar curvature.  Since we are interested in the case of $S^1$ fibers, this implies that it is impossible to have multiple Einstein metrics in the canonical variation of $\pi$ when one considers $S^1$ fibers.  Since the vector field $X$ in the definition of a quasi-Einstein metric provides an extra degree of freedom of possible metrics, the fact that the analogous result for Einstein metrics is more rigid agrees with our intuition. 

\section{Sasakian Manifolds and quasi-Einstein metrics}

In this section, we will explore a relationship between Sasakian manifolds and quasi-Einstein metrics in dimension three.  We will then proceed to use a classification result due to Tanno \cite{Tanno69} of Sasaskian manifolds of constant $\phi$-sectional curvature to obtain a classification result about quasi-Einstein metrics in dimension three, which will imply Theorem \ref{thm:1}.  See also \cite{Ghosh19} for additional results between $K$-contact structures, quasi-Einstein metrics, and Sasakian manifolds.

\begin{defn}
    A Riemannian manifold $(M,g)$ is called \textbf{Sasakian} if its metric cone $g = (M \times \mathbb{R}, \: t^2g + dt^2)$ is K\"ahler with K\"ahler form $t^2d\theta + 2t \, dt \wedge \theta$ where $\theta$ is a one form on $(M,g)$.
\end{defn}

In the case we are considering, there is a very useful characterization for when a manifold $(M,g)$ is Sasakian, which is proved in Proposition 1.1.2 of \cite{Boyer01}.

\begin{prop} [\cite{Boyer01}]
    $(M, g)$ is Sasakian if and only if there exists a unit length Killing field $\xi$ on $M$ such that the $(1,1)$-tensor $\Phi$ defined by $\phi(Y) = -\nabla_Y \xi$ satisfies the following property:
        \[(\nabla_Y \phi)(Z) = g(\xi, Z)Y - g(Y, Z)\xi \]
    \noindent where $Y, Z \in \mathfrak{X}(M)$.
\end{prop}

\noindent In the context of the above proposition, the vector field $\xi$, the dual one form $\eta = g(\cdot, \xi)$, and the $(1,1)$ tensor $\phi$ care called the \textit{Sasakian structure} of the Sasakian manifold $(M,g, \phi, \xi, \eta)$.  From now on, when we refer to a Sasakian manifold $(M,g)$, we also will adopt these definitions for $\eta, \phi, \xi$.  

In dimension three, which is the case of primary interest for this section, an even more useful characterization is true.  This is a well known result, but we provide a proof below for the sake of completion.

\begin{lem} 
    Let $(M,g)$ be a Riemannian manifold of dimension three.  Then $(M,g)$ is Sasakian if and only if there exists a unit Killing field $\xi$ such that all sectional curvatures containing $\xi$ are equal to one. 
\end{lem}

\begin{proof}
    Let $\{Y, Z, \xi\}$ be an orthonormal basis at an arbitrary point $p \in M$. It is shown in Proposition 8.1.3 of \cite{Petersen16} that $(\nabla_Y \phi)(Z)=R(\xi,Y,Z)$ where $\phi$ is defined as in the previous proposition.  By Proposition 8.2.1(2) of \cite{Petersen16}, we have the following:
        \[\mathrm{Hess} \left (\frac{1}{2}|\xi|^2 \right )(Y,Y)=|\nabla_Y \xi|^2-R(\xi, Y, Y, \xi)\]

    \noindent Furthermore, since $|\xi|$ is constant, $\mathrm{Hess} \left (\frac{1}{2}|\xi|^2 \right) = 0$ and so by the above, one has that $|\nabla_Y \xi|^2=|\phi(Y)|^2=R(\xi, Y, Y, \xi) = 1$.  

    To show $(M,g)$ is Sasakian, it suffices to show that $(\nabla_Y \phi)(Z) = g(\xi, Z)Y - g(Y, Z)\xi$ as in the above proposition.  Define a tensor $T$ by $T(W,X,Y,Z) = g(g(X,Y)W-g(W,Y)X,Z)$.  We want to show $R(\xi, Y, Z, W) = T(\xi, Y, Z, W)$ to prove the lemma.  

    We first observe that $T$ and $R$ possess the same symmetries.  Next, note that $\phi$ is a skew-symmetric map and $|\phi(Y)|^2=1$.  In dimension three, this means that there exists an orthonormal basis $\{E_1, E_2, \xi\}$ such that $\phi(E_1) = -E_2$ and $\phi(E_2) = E_1$. Furthermore, we have $\phi(\xi) = \nabla_{\xi}\xi = 0$.

    We then observe that $R(\xi, E_1,E_1, E_1)=T(\xi, E_1, E_1, E_1) =0$. Furthermore, we have that 
    
    \begin{align*}
        R(\xi, E_1,E_1, E_2) &= g((\nabla_{E_1}\phi)(E_1), E_2) 
        \\&= g(\nabla_{E_1}(\phi(E_1)) - \phi(\nabla_{E_1}E_1),E_2)
        \\&= g(-\nabla_{E_1}E_2 - \phi(\nabla_{E_1}E_1),E_2)
        \\&= g(-\nabla_{E_1}E_2, E_2) - g(\phi(\nabla_{E_1}E_1),E_2)
        \\&= -\frac{1}{2}D_{E_1}|E_2|^2 +g(\nabla_{E_1}E_1,\phi(E_2))
        \\&= g(\nabla_{E_1}E_1,\phi(E_2))
        \\&= g(\nabla_{E_1}E_1,E_1)
        \\&= \frac{1}{2}D_{E_1}|E_1|^2=0.
    \end{align*}
    
    \noindent Similarly, we have that
    
    \begin{align*}
        T(\xi, E_1,E_1, E_2) &= g(g(E_1,E_1)\xi-g(\xi,E_1)E_1,E_2) = 0.
    \end{align*}

    Hence, $R(\xi, E_1,E_1, E_2)=T(\xi, E_1, E_1, E_2) =0$.  It follows from these facts that $R(\xi, Y, Z, W) = T(\xi, Y, Z, W)$, and so $(M,g)$ is Sasakian as desired.

    Conversely, suppose that $(M,g)$ is Sasakian.  Then $(M,g)$ admits a unit length Killing field $\xi$.  Since $(M,g)$ is Sasakian, $g(\xi, Y)Z - g(Y, Z)\xi = (\nabla_Y \phi)(Z) = R(\xi, Y, Z)$.  It easily follows that all sectional curvatures involving $\xi$ are equal to one, and so the claim is proved.  
\end{proof}

Using this above lemma, it is relatively straightforward to state and prove the following result, which relates $m$-quasi Einstein metrics in dimension three to Sasakian manifolds.

\begin{prop}
    Let $(M, g, X)$ be a triple satisfying the quasi-Einstein equation where $(M, g)$ is closed, $M$ is dimension three, and $(M,g)$ has constant scalar curvature.  Suppose further that $\lambda + \frac{|X|^2}{m} = 2$. Then $(M, g)$ is Sasakian, where we take $\xi = \frac{X}{|X|}$. 
\end{prop}

\begin{proof}
    \noindent Since $(M, g)$ is compact with constant scalar curvature, and since $(M, g, X)$ is quasi-Einstein, $X$ must be a Killing field, as proved in Theorem 4.2 of \cite{Ghosh20}.  Furthermore, this Killing field must have constant norm, as shown in \cite{Bahuaud24}.  Thus, $\xi = \frac{X}{|X|}$ is a unit length Killing field on $(M,g)$.

    \noindent Let $\{\xi, E_0, E_1\}$ be an orthonormal basis at a point $p \in M$.  The quasi-Einstein equations, in conjunction with the definition of sectional curvature in dimension three, yield the following:
        \begin{align}    
            \lambda + \frac{|X|^2}{m} &= \mathrm{Ric}(\xi, \xi) = \mathrm{Sec}(\xi, E_0) + \mathrm{Sec}(\xi, E_1). \\
            \lambda &= \mathrm{Ric}(E_0, E_0) = \mathrm{Sec}(\xi, E_0) + \mathrm{Sec}(E_0, E_1). \\
            \lambda &= \mathrm{Ric}(E_1, E_1) = \mathrm{Sec}(\xi, E_1) + \mathrm{Sec}(E_0, E_1).  \\ \nonumber
        \end{align}

\noindent Subtracting the last two equations from one another imply $\mathrm{Sec}(\xi, E_0) = \mathrm{Sec}(\xi, E_1)$.  By assumption, $\lambda + \frac{|X|^2}{m} = 2$, which imply $\mathrm{Sec}(\xi, E_0) = \mathrm{Sec}(\xi, E_1) = 1$.  Since the orthonormal basis we chose was arbitrary (apart from $\xi$), Lemma 2.3 immediately implies $(M,g)$ is Sasakian.
\end{proof}

We will now observe that given a quasi-Einstein metric $(M, g, X)$, the metric can be rescaled by a constant factor to give another quasi-Einstein metric so that the condition $\lambda + \frac{|X|^2}{m} = 2$ is satisfied.  This will follow from the next two propositions.

\begin{prop}
    Let $(M, g, X)$ be a triple satisfying the quasi-Einstein equation where $(M,g)$ is closed, and $M$ has constant scalar curvature.  Then either:
        \begin{itemize}
            \item[(a)] $dX^* = 0$ and $(M, g)$ splits as a product $(N \times S^1, g_N + d\theta^2)$ where $N$ is Einstein, or 
            \item[(b)] $\lambda + \frac{|X|^2}{m} > 0$.
        \end{itemize}
\end{prop}

\begin{proof}
    Since $X$ is Killing (again by \cite{Ghosh20}), it satisfies the following Bochner formula (as presented in Proposition 8.2.1 of \cite{Petersen16}):
        \[ \frac{1}{2}\Delta|X|^2 = |\nabla X|^2 - \mathrm{Ric}(X, X). \]
    Since $(M, g, X)$ is quasi-Einstein with constant scalar curvature, $X$ is constant length (as mentioned in the introduction), so clearly $\frac{1}{2}\Delta|X|^2 = 0$.  Furthermore, using the fact that $X$ is Killing and $(M, g, X)$ is quasi-Einstein, $\mathrm{Ric}(X, X) = \lambda |X|^2 + \frac{1}{m}|X|^4$.  We therefore have that $\mathrm{Ric}(K, K) = \lambda + \frac{1}{m}|X|^2$ where $K = \frac{X}{|X|}$.
    Since $K$ is also Killing with constant length, it follows that \[ 0 = \frac{1}{2}\Delta|K|^2 = |\nabla K|^2 - \mathrm{Ric}(K, K) = |\nabla K|^2 - \lambda - \frac{1}{m}|X|^2 \]  and hence,
    \[ \lambda + \frac{1}{m}|X|^2 = |\nabla K|^2 \geq 0.  \] \\

    \noindent If $|\nabla K|^2 \neq 0$, then clearly $\lambda + \frac{|X|^2}{m} > 0$, and we are in case (b) in the proposition.  Hence, we suppose $|\nabla K|^2 = 0$ (so of course $|\nabla X|^2 = 0$ as well).  Knowing $X$ is parallel, this gives us a local splitting of $(M, g)$ as $(N \times S^1, g_N + d\theta^2)$ where $g_N$ is an Einstein metric on $N$.  It also follows from the quasi-Einstein equation that $X = \sqrt{-m\lambda}U$ where $U$ is a unit vector in the $S^1$ direction.  Following an argument in \cite{Wylie23}, it follows that this splitting must in fact be global, by considering the corresponding quasi-Einstein metric in the universal cover $\tilde{M}$ and using the fact that the eigenspaces of the Ricci tensor are preserved under isometries.  Finally, the fact that $X$ is Killing and satisfies $|\nabla X| = 0$ immediately implies $dX^* = 0$, proving the claim.
    
\end{proof}

Alternatively, Proposition 2.5 easily follows from Equation 1.3 of \cite{Bahuaud24}.  This states that for $(M,g,X)$ compact and quasi-Einstein, we have that

    \begin{align} \label{eq:2.5}
        |dX^{*}| = 4|X|^2\left ( \lambda + \frac{|X|^2}{m}\right ).
    \end{align}

\noindent If $|dX^{*}|=0$, $X^{*}$ is closed, and we are in case (a), which is handled in \cite{Wylie23}.  If $|dX^{*}| > 0$, then it follows immediately from (\ref{eq:2.5}) that $\lambda + \frac{|X|^2}{m} > 0$ as claimed.  The author of this paper thanks Alcides de carvalho Júnior for this observation.

\begin{prop}
    Let $(M, g, X)$ be a triple satisfying the quasi-Einstein equation, $\mathrm{dim}(M) = 3$, and $M$ has constant scalar curvature.  Then the triple $(M, \tilde{g}, \tilde{X})$ is quasi-Einstein, where $\tilde{g} = t^2g$, $t > 0$ is a constant, and $\tilde{X} = \frac{X}{t^2}$.  Furthermore, there exists such a choice of $t$ such that $\tilde{\lambda} + \frac{\tilde{X}}{t^2} = 2$ where $\tilde{\lambda}$ is the $\lambda$ associated to $(M, \tilde{g}, \tilde{X})$.  In particular, the triple $(M, \tilde{g}, \tilde{X})$ is Sasakian.  
\end{prop}

\begin{proof}
    Let $(M, g, X)$ be quasi-Einstein.  Consider the metric $\tilde{g} = t^2g$ for some $t > 0$ and arbitrary smooth vector fields $Y, Z \in \mathfrak{X}(M)$.  Using the scalar invariance of the $(0, 2)$ Ricci tensor, we have that
        \begin{align*}
            g(\mathrm{Ric}(Y), Z) = \mathrm{Ric}(Y, Z) = \widetilde{\mathrm{Ric}}(Y, Z) = \tilde{g}(\widetilde{\mathrm{Ric}}(Y), Z) = t^2g(\widetilde{\mathrm{Ric}}(Y), Z).
        \end{align*}
    From this, we can therefore conclude that, as a $(1,1)$ tensor, we have that 
        \begin{align*}
            \widetilde{\mathrm{Ric}}(Y) = \frac{1}{t^2}\mathrm{Ric}(Y).
        \end{align*}

    \noindent Since $(M, g, X)$ is quasi-Einstein of constant scalar curvature by assumption, we have that
        \begin{align*}
            \mathrm{Ric} = \lambda g + \frac{X^* \otimes X^*}{m}
        \end{align*}
    where we are using the fact that constant scalar curvature implies $X$ Killing (again by \cite{Ghosh20}).  Defining $\tilde{\lambda} = \frac{\lambda}{t^2}$ and $\tilde{X} = \frac{X}{t^2}$, it is straightforward to verify that
        \begin{align*}
            \widetilde{\mathrm{Ric}} = \tilde{\lambda}\tilde{ g} + \frac{\tilde{X}^* \otimes \tilde{X}^*}{m}
        \end{align*}
    for any $t > 0$ and so the triple $(M, \tilde{g}, \tilde{X})$ is $m$-quasi Einstein.

    \noindent Setting $t = \sqrt{2(\lambda+\frac{|X|^2}{m})^{-1}}$ (which is defined by the previous proposition), we see that $\tilde{\lambda}+\frac{|\tilde{X}|^2}{m} = 2$ and so $(M, \tilde{g}, \tilde{X})$ is Sasakian by Proposition 2.3 as claimed (note here that $|\tilde{X}|$ is being taken with respect to the metric $\tilde{g}$).
\end{proof}

By the previous three propositions, any $m$-quasi Einstein manifold with $dX^* \neq 0$ can be rescaled by an appropriate factor to be Sasakian.  Three dimensional Sasakian manifolds are well understood, have been classified (see \cite{Belgun00}, \cite{Tanno69}, and \cite{Boyer06}).  Therefore, to understand the classification of closed quasi-Einstein manifolds of dimension three of constant scalar curvature, we will refer to this classification.  Specifically, we will refer to Tanno's classification, but we need to define a few terms first. \\

\begin{defn}
    A Sasakian manifold is called \textbf{$\eta$-Einstein} if $\mathrm{Ric} = \lambda g +\nu \eta\otimes \eta$ for constants $\lambda$ and $\nu$.
\end{defn}

\noindent Observe that a three dimensional quasi Einstein metric with constant scalar curvature is $\eta$-Einstein with $\nu = \frac{|X|^2}{m}$ and $\eta = \frac{X^*}{|X|}$.  The relationship between $\eta$-Einstein metrics and quasi-Einstein metrics is also studied in \cite{Ghosh19}.

\begin{defn} 
    A Sasakian manifold $(M, g, \phi, \xi, \eta)$ is said to have \textbf{constant $\phi$-sectional curvature} if $K(Y, \phi Y)$ is constant, where $K(Y, \phi Y)$ denotes the sectional curvature spanned by the vectors $Y$ and $\phi Y$ at a point $p \in M$ and $Y$ is some vector such that $\eta(Y) = 0$.
\end{defn}

\noindent Now, we turn to Proposition 30 in \cite{Boyer06} which connects these two notions in dimension three.

\begin{prop}(\cite{Boyer06})
    A three dimensional Sasakian manifold is $\eta$-Einstein if and only if it has constant $\phi$-sectional curvature.
\end{prop}

\noindent In other words, using this proposition, we know in the case we are interested in, quasi-Einstein metrics of constant scalar curvature in dimension three are Sasakian with constant $\phi$-sectional curvature.  This provides an appropriate setup for the following result due to Tanno (\cite{Tanno69}), which will have us on the cusp of obtaining a quasi-Einstein classification we desire.  First, we define the notion of $\mathcal{D}$-homothety, which is used in Tanno's classification.

\begin{defn}
    Consider a Sasakian structure $(M,g, \phi, \xi,\eta)$.  We say that the Sasakian structure $(M, g', \phi, \xi', \eta')$ is \textbf{$\mathcal{D}$-homothetic} to $(M,g, \phi, \xi,\eta)$ if there exists a positive constant $t$ such that $g' = tg + (t^2-t)\eta \otimes \eta$, $\xi' = t^{-1}\xi$, $\eta' = t\eta$ and $\phi' = \phi$.
\end{defn}

\begin{prop}(\cite{Tanno69})
    Let $(M, g)$ be a dimension $2n+1$, complete, simply connected Sasakian manifold with constant $\phi$-sectional curvature $H$.Then $(M,g)$ is $\mathcal{D}$-homothetic to a Sasakian structure $(M,g')$ with $H=1, -3, -4$.  Furthermore,
    \begin{itemize}
        \item[(a)] If $H = 1$ $(M,g') \cong S^{2n+1}$ with the round metric.
        \item[(b)] If $H = -3$, then $(M,g') \cong \mathbb{R}^{2n+1}[H]$.
        \item[(c)] If $H = -4$, then $(M,g') \cong (L, CD^n)[H]$.
    \end{itemize}
\end{prop}

Here, $S^{2n+1}[H]$ denotes the sphere with the standard round metric of $\phi$-sectional curvature $H$, and $(L, CD^n)$ denotes a line bundle over a homogeneous complex space $CD^n$.  In the three dimensional case, $n=1$, and as noted earlier, $\phi$-sectional curvature is always constant.  Letting $(\tilde{M}, \tilde{g}, \tilde{X})$ be the corresponding quasi-Einstein metric on the universal cover of a quasi-Einstein metric $(M, g, X)$ (showing the quasi-Einstein condition lifts this way to the universal cover is straightforward), and putting everything together, we obtain the following:

\begin{prop}
    Let $(M, g, X)$ be a dimension three quasi-Einstein metric of constant scalar curvature.  Let $(\tilde{M}, \tilde{g}, \tilde{X})$ be the corresponding quasi-Einstein metric in the universal cover of $M$.  Then exactly one of the following holds:
    \begin{itemize}
        \item[(a)] $M$ is spherical, and $M$ is $\mathcal{D}$-homothetic to some Sasakian structure $M'$ such that $\tilde{M'} \cong S^{3}$ with the round metric.
        \item[(b)] $M$ has Thurston geometry $\mathrm{Nil}$.
        \item[(c)] $M$ has Thurston geometry $\widetilde{SL}_2$.
        \item[(d)] $dX^* = 0$ and $(M,g)$ splits globally as $(N \times S^1, g_N+d\theta^2)$ for a constant curvature surface $N$.
    \end{itemize}
\end{prop}

With the above classification, we are nearly ready to prove Theorem 1.3, showing that all quasi Einstein metrics in three dimensions of constant scalar curvature are locally homogeneous.  Before completing the proof, we show that homogeneous is preserved under $\mathcal{D}$-homotheties.

\begin{lem}
    Let $(M,g,X)$ be compact quasi-Einstein metric of dimension three where $(M,g)$ has constant scalar curvature, $dX^* \neq 0$, and let $(M, g, \phi, \xi, \eta)$ be the Sasakian structure associated to $(M,g,X)$.  Suppose further that $(M,g)$ is homogeneous. Then any $\mathcal{D}$-homothety of this structure $(M, g_t, \phi, \xi_t, \eta_t)$ as defined above is still homogeneous.
\end{lem}

\begin{proof}
    Let $x, y \in M$ be arbitrary.  Since $(M, g)$ is homogeneous by assumption, there is an isometry $\Phi: (M, g) \to (M,g)$ such that $\Phi(x) = y$.  First, we show that the one form $\eta$ is invariant under $\Phi$ on $(M, g_t)$ (up to a possible sign).  To see this, we recall that $\eta(Y) = g(\xi, Y)$ for an arbitrary vector field $Y$.  We also recall, as mentioned above, that eigenvalues of the Ricci tensor are invariant under isometries.  Since $(M,g,X)$ has constant scalar curvature, we have that $\lambda + \frac{|X|^2}{m}$ is an eigenvalue of the Ricci tensor associated to the eigenvector $X$.  That is, we have that $\mathrm{Ric}(Y, Z) = \mathrm{Ric}(\Phi(Y), \Phi(Z))$ for $Y, Z \in \mathfrak{X}(M)$.  In particular, this means that $\mathrm{Ric}(\xi, \xi) = \mathrm{Ric}(d\Phi(\xi), d\Phi(\xi)) = \lambda + \frac{|X|^2}{m}$ (where, again, $\xi = \frac{X}{|X|}$).  Since the eigenvector $\xi$ has multiplicity one, we conclude that $\xi = \pm d\Phi(\xi)$. \\      
    
    Next, We compute that for arbitrary vectors $v, w \in TM$, one has that
        \begin{align*}
            \\
            g_t(v,w) &= tg(v,w) + (t^2-t)(\eta \otimes \eta)(v,w) \\&= 
            tg(d\Phi(v),d\Phi(w)) + (t^2-t)g(\xi, v)g(\xi,w) \\&= 
            tg(d\Phi(v),d\Phi(w)) + (t^2-t)g(d\Phi(\xi), d\Phi(v))g(d\Phi(\xi),d\Phi(w)) \\&=
            tg(d\Phi(v),d\Phi(w)) + (t^2-t)g(\pm\xi, d\Phi(v))g(\pm\xi,d\Phi(w)) \\&=
            tg(d\Phi(v),d\Phi(w)) + (t^2-t)(\eta \otimes \eta)(d\Phi(v),d\Phi(w)) \\&=
            g_t(d\Phi(v),d\Phi(w)).
        \end{align*}
    In the fourth line of the above, the two instances of $\pm \xi$ must agree in sign.  Either way, the fifth equality is valid since $(-\eta) \otimes (-\eta) = \eta \otimes \eta$.  Hence, $\Phi$ is an isometry with respect to the metric $(M, g_t)$ as well, and so the  $\mathcal{D}$-homothety is still homogeneous as desired.
\end{proof}

Using the above Proposition, we can finish the proof of Theorem 1.3, showing that all closed $m$-quasi Einstein metrics in dimension three of constant scalar curvature must be locally homogeneous.  

\begin{proof}[Proof of Theorem \ref{thm:1}] 
    Let $(M, g, X)$ be a closed, dimension three $m$-quasi Einstein metric of constant scalar curvature.  By the work above, we know that either $dX^*=0$, or $(M, g, X)$ is Sasakian.  In the former case, we know by Proposition 2.5 and \cite{Wylie23} that we must obtain a global splitting $(M, g) = (N \times S^1, g_N +d\theta^2)$ with $N$ Einstein.  As a product of locally homogeneous spaces, $(M, g)$ is clearly locally homogeneous in this case. \\

    Now suppose $dX^* \neq 0$.  Now, by Proposition 2.6, we know that (up to a scale factor) $(M, g, X)$ admits a Sasakian structure $(M, g, \phi, \xi, \eta)$ with $\xi = \frac{X}{|X|}$, and $\eta = \frac{1}{|X|}X^*$.  Passing to the universal cover $(\tilde{M},g, \tilde{X})$, we can refer to Tanno's classification (Proposition 2.11).  We observe that all of the metrics in parts (a), (b) and (c) are homogeneous.  By Lemma 2.13, homogeneity is preserved under $\mathcal{D}$-homotheities.  Therefore, we can conclude that $(\tilde{M},g,\tilde{X})$ is a homogeneous space, and by a result due to Singer in \cite{Singer60}, this tells us that $(M,g,X)$ must be locally homogeneous as claimed.
\end{proof}



\begin{rmk}
Revisiting Lim's classification of three dimensional locally homogeneous quasi-Einstein metrics, we observe that Lim found that nontrivial locally homogeneous quasi-Einstein metrics in dimension three must have one of the following Thurston geometries: $SU(2), Nil, \widetilde{SL_2}, H^2 \times \mathbb{R}$ or $S^2 \times \mathbb{R}$.  In the first case, $SU(2)$, we recall that $SU(2)$ is diffeomorphic to $\mathbb{S}^3$.  The examples Lim found in this case are known as $\textbf{Berger spheres}$.  The $\mathcal{D}$-homotheities referred to in Tanno's classification correspond to shifting the parameter of the Berger sphere. \\

The $H^2 \times \mathbb{R}$ and $S^2 \times \mathbb{R}$ examples in Lim's paper correspond to cases where $dX^* = 0$.  These cases are not Sasakian.
\end{rmk}

\section{Riemannian Submersions in Higher Dimensions}

With the three dimensional case taken care of, the goal of this section will be to study examples of quasi-Einstein metrics of constant scalar curvature in higher dimensions.  To achieve this, we will first suppose $(M,g,X)$ is quasi-Einstein with constant scalar curvature, and that the integral curves generated by $X$ are closed.  We further assume that the action of $S^1$ on $M$ induced by $X$ is a free action, and so there is a Riemannian submersion $\pi: M \to B \cong M/S^1$, where $B$ is a manifold.  Observe that by previous results in \cite{Bahuaud24} mentioned earlier, we know that $|X|$ is constant, and the integral curves of $X$ are geodesics (i.e. the $S^1$ fibers of $\pi$ are totally geodesic).   \\ 

To fix notation, we decompose $TM = \mathcal{V} \oplus \mathcal{H}$ where $\mathcal{V}$ is the vertical distribution of $\pi$, defined to be $\mathrm{ker}(d\pi)$, and $\mathcal{H}$ the horizontal distribution, the orthogonal complement.  Following \cite{Besse87}, we define the tensor $A$ as $A_{E_1} E_2 = \mathcal{H}\nabla_{\mathcal{H}E_1}\mathcal{V}E_2 + \mathcal{V}\nabla_{\mathcal{H}E_1}\mathcal{H}E_2$ for $E_1, E_2 \in \mathfrak{X}(M)$ where we use $\mathcal{V}$ and $\mathcal{H}$ to denote projections onto the vertical and horizontal distributions of $\pi$ respectively.  Using formulae due to O'Neill, one can compute the Riemann curvature tensor of the total space $M$ of a Riemannian submersion in terms of the curvature of the base $B$, the tensor $A$ defined above, and a tensor $T$ describing the curvature of the fibers of $\pi$ (but note here that $T \equiv 0$ since our fibers are totally geodesic).  From these formulae and the quasi-Einstein equation, one can easily obtain the below, which is an easy modification of Proposition 9.61 in \cite{Besse87}.  First, we will recall the notion of a \textit{Yang-Mills connection}.

\begin{defn}
    Let $\pi: (M,g) \to B$ be a Riemannian submersion, and let $\mathcal{H}$ be the horizontal distribution.  We say $\mathcal{H}$ is a \textbf{Yang-Mills connection} if
        \[ g(\check{\delta}A(Y), U)-g(A_Y, T_U) = 0. \]
    \noindent where $\displaystyle \check{\delta}A = -\sum_{i}(\nabla_{Y_i}A)_{Y_i}$ and $\{Y_i\}$ is an orthonormal basis for the horizontal distribution at a point.
\end{defn}

\noindent Observe that in the case we are interested in, namely the case where the fibers are totally geodesic $S^1$ fibers, the Yang-Mills condition is equivalent to the condition that $\check{\delta}A = 0$.

\noindent Next, we state Proposition 9.61 in \cite{Besse87}.

\begin{prop}[\cite{Besse87}, Prop. 9.61]
    Let $\pi: (M,g) \to (B,\check{g})$ be a Riemannian submersion with totally geodesic fibers.  Then $(M,g)$ is Einstein if and only if the following conditions are satisfied:
        \begin{align*}
            &\mathcal{H} \textrm{ is a Yang-Mills connection}\\
            &\widehat{Ric}(U, V) + g(AU, AV) = \lambda g(U,V)  \\
            &\widecheck{Ric}(Y, Z) - 2g(A_Y, A_Z) = \lambda g(Y, Z)
        \end{align*}
    where $g(A_Y, A_Z) = g(A_Y N, A_Z N)$, and $\widecheck{Ric}$ denotes the Ricci tensor on $B$,  $\widecheck{Ric}$ denotes the Ricci tensor on the fibers of $\pi$, and $g(AU,AV) = \sum_i g(A_{Y_i}U, A_{Y_i}V)$ where $\{Y_i\}$ is an orthonormal basis for the horizontal distribution at a point.
\end{prop}

\noindent Now, we state and prove the analogous result for quasi-Einstein metrics in the case where the fibers are $S^1$.

\begin{prop}
    Let $(M,g)$ have constant scalar curvature, and let $\pi: (M,g) \to (B,\check{g})$ be a Riemannian submersion with totally geodesic, $S^1$ fibers.  Furthermore, let $X$ be a smooth vector field tangent to the fibers of $\pi$.  The triple $(M,g,X)$ is quasi-Einstein if and only if the following conditions are satisfied:
        \begin{align}
            &\mathcal{H} \textrm{ is a Yang-Mills connection}\\
            &|A|^2 = \lambda + \frac{|X|^2}{m} \label{eq:3.2}  \\
            &\widecheck{Ric}(Y, Z) - 2g(A_Y, A_Z) = \lambda g(Y, Z)
        \end{align}
    where $|A|^2 = \sum g(A_{Y_i}, A_{Y_i})$ and $\{Y_i\}$ is an orthonormal basis of $\mathcal{H}$, and $ g(A_{Y}, A_{Z})$ is defined in the previous proposition.
\end{prop}

\begin{proof}
     $\mathcal{H}$ being a Yang-Mills connection means $\check{\delta} A = 0$ which equivalent to $\widecheck{Ric}(Y, U) = 0$ for $U \in \mathcal{V}$ and $Y \in \mathcal{H}$.  The other two equations come from plugging in horizontal and vertical vectors into the quasi-Einstein conditions where the Ricci curvatures on the left hand sides are written in terms of $A$ and the Ricci tensor on the base.  Observe, for instance, that the term $\widehat{Ric}$ vanishes in this case since our fibers here are one-dimensional.  Observe further that to obtain (\ref{eq:3.2}), we use the fact that $X$ is Killing, which follows from the results mentioned in the introduction. 
\end{proof}

\noindent From the above, it is not too difficult to obtain the following useful lemma:

\begin{lem}
    Let $(M,g)$ have constant scalar curvature, $\pi: (M,g) \to (B,\check{g})$ be a Riemannian submersion with totally geodesic, $S^1$ fibers and suppose $X$ is a smooth vector field tangent to the fibers of $\pi$.  If $(M,g,X)$ is quasi-Einstein, then $|A|$ is constant and $B$ has constant scalar curvature.
\end{lem}

\begin{proof}
    By results mentioned earlier, we know that $|X|$ is constant, hence $\lambda + \frac{|X|^2}{m}$ is constant.  By (3.2), it then immediately follows that $|A|$ is constant as well.\\

    To show $B$ has constant scalar curvature, consider the trace of (3.3), which is
        \[\check{s} - 2|A|^2 = \lambda n. \]
    Since we know $|A|$ is constant, it clearly follows that $\check{s}$ is constant as well.
\end{proof}

With the above established, we are now interested in principal $S^1$ bundles $p: P \to B$ over a compact Einstein base $B$.  Starting with a metric $\check{g}$ on $B$, our goal will be to study quasi-Einstein metrics on $P$ where $p: P \to B$ is a Riemannian submersion with totally geodesic fibers.  We start by recalling some terms.

\begin{defn}
    A \textbf{principal G-bundle} $p: P \to B$ is a locally trivial fiber bundle which is the quotient of a free action of the Lie group $G$ on $P$.  A \textbf{principal connection} is a one form $\theta$ on $P$ which takes values in the Lie algebra $\mathfrak{g}$ and $\theta$ sends a vector field $V \in P$ to the vector field in $\mathfrak{g}$ such that
        \begin{itemize}
            \item[(a)] $\mathrm{Ad}_g(\theta \circ g) = \theta$ for all $g\in G$ and
            \item[(b)] $\theta(\hat{V}) = V$ where $\hat{V}$ is the vector field generating the one-parameter family of isomorphisms generated by $\mathrm{Exp}(tV)$ of $G$.
        \end{itemize}
    Furthermore, we define the \textbf{curvature} of $\theta$ to be $\Omega(Y, Z) = d\theta(Y,Z) + \theta \wedge \theta$.  
\end{defn}

\begin{rmk}
    In the case where $G = S^1$, observe that $\theta \wedge \theta = 0$ and so the curvature form $\Omega$ associated to $\theta$ is the pullback of a closed 2-form on $B$.  Hereinafter, we define $\omega = d\theta$.
\end{rmk}

The following result, proved by \cite{Vilms70} and also presented as Theorem 9.59 in \cite{Besse87}, connects principal circle bundles with Riemannian submersions with totally geodesic fibers.

\begin{thm}[\cite{Vilms70}] \label{Vilms}
    Let $G$ be a Lie group and let $p: P \to B$ be a principal $G$ bundle.  Given a Riemannian metric on $\check{g}$ on $B$ and a principal connection on $\theta$ on $P$, there exists a unique metric $g$ on $P$ such that $p$ is a Riemannian submersion with totally geodesic fibers.
\end{thm}

As observed in Equation 9.65 in \cite{Besse87}, the tensor $A$ and the 2-form $\omega$ (as defined in the remark above) are related by

\begin{equation}
    A_Y Z = -\frac{1}{2}\omega(\check{Y}, \check{Z})\check{U}
\end{equation}

\noindent under the assumption that the fibers of the submersion have length $2\pi$, and where checks refer to the pullbacks of vector fields to $P$ and $\check{U}$ is the unit vector field on $P$ generating the $S^1$ action.  Observe that the choice of having the fibers have length $2\pi$ is equivalent to a choice of a principle connection on $P$.  It then follows by Theorem \ref{Vilms} that there exists a unique metric $g$ on $P$ so that $p$ is a Riemannian submersion with totally geodesic fibers.  Using equations (3.1)-(3.3) along with the above allow us to establish the following:

\begin{prop}
    Let $p: P \to B$ be a principal $S^1$ bundle with $B$ compact and a Riemannian submersion with fibers length $2\pi$, and as above, let $\omega$ be the pullback on $B$ of the curvature form of the principal connection $\theta$.  Let $g$ be the unique metric on $P$ such that $p$ is a Riemannian submersion with totally geodesic fibers.  Suppose $(P,g)$ is of constant scalar curvature.  Suppose further that $X$ is a smooth vector field on $M$ such that $X$ is tangent to the fibers of $\pi$ (i.e. $X \in \mathcal{V}$).  Then $(P,g,X)$ is quasi-Einstein if and only if
        \begin{itemize}
            \item[(a)] $\omega$ is coclosed (it follows that $\omega$ is closed and has constant norm).
            \item[(b)] $\widecheck{Ric}(Y, Z) - \frac{1}{2}g(\omega_y, \omega_z) = (\frac{|\omega|^2}{4}-\frac{|X|^2}{m})g(Y, Z)$ where
            $g(\omega_y, \omega_z) = \sum_i \omega(Y,X_i)\omega(Z,X_i)$ for an orthonormal basis $\{X_i\}$.
        \end{itemize}
\end{prop}

\begin{proof}
    $\omega$ being coclosed means that $\delta\omega = 0$, where $\delta$ is the formal adjoint of the exterior derivative, $d$.  This is equivalent to $\check{\delta}A = 0$, which, according to the Ricci curvature formulae for submersions, is further equivalent to $\widecheck{Ric}(Y, U) = 0$ for $U \in \mathcal{V}$ and $Y \in \mathcal{H}$.  The second equation follows from solving for $\lambda$ in (3.2), writing $|A|$ in terms of $|\omega|$ using (3.4), and plugging in horizontal vectors into the quasi-Einstein equation.  
\end{proof}

Assuming a compact Einstein base, we are now ready to prove our next main result, Theorem \ref{thm:2}, which provides a necessary and sufficient condition for the existence of a circle bundle over a compact Einstein base admitting a quasi-Einstein metric.  We begin by recalling the definition of almost K\"ahler.

\begin{defn}
    An \textbf{almost K\"ahler structure} on $M$ is an almost Hermitian structure on $M$ such that $\omega$ is closed. 
\end{defn}

\begin{proof}[Proof of Theorem \ref{thm:2}] 
    Since $B$ is Einstein by assumption, we can let $\widecheck{Ric}(Y, Z) = \check{\lambda}\check{g}$.  Rearranging the equation in Proposition 3.8(b), we obtain $g(\omega_Y, \omega_Z) = (2\check{\lambda} - \frac{|\omega|^2}{2} + \frac{2|X|^2}{m})g(Y, Z)$.  Positive definiteness of the metric and $\omega$ shows that $|\omega| = 0$ if and only if $\check{\lambda} + \frac{|X|^2}{m}  = 0$.  If $|\omega| = 0$, then the the circle bundle is trivial, and so $P = B \times S^1$ as claimed, and this takes care of part (a).  \\

    For part (b), suppose that $\check{\lambda} + \frac{|X|^2}{m} > 0$.  In this case, define $J$ as $\omega(Y, Z) = \check{g}(JY, Z)$.  We have that

    \begin{align}
        & g(\omega_Y, \omega_Z) = \sum_i \omega(Y, Y_i)\omega(Z, Y_i) = 
        \sum_i \check{g}(JY, Y_i)\check{g}(JZ, Y_i) = \check{g}(JY, JZ)  \notag
        \\&= \omega(Y, JZ) = -\omega(JZ, Y) = -\check{g}(J^2Z, Y) =  \left (2{\check{\lambda}} - \frac{|\omega|^2}{2} + \frac{2|X|^2}{m}\right )g(Y, Z).
    \end{align}

    From this, it is evident that an appropriate multiple of $\omega$ has the property that $J^2 = -1$.  In other words, there exists an almost complex structure on $B$, $\omega'$, which differs from $\omega$ by a constant factor.  \\

    Conversely, consider an almost K\"ahler structure $(J, \check{g}, \omega)$ on $P$ such that $[\omega]$ is a constant multiple of $\alpha \in H^2(B, \mathbb{R})$.  Then, one can choose a constant $c$ so that 
        \begin{align}
        & g(\omega_Y, \omega_Z) =  -\check{g}(J^2Z, Y) = \left (2{\check{\lambda}} - \frac{|\omega|^2}{2} + \frac{2|cX|^2}{m} \right )g(Y, Z)
    \end{align}
    \noindent is satisfied, so that condition (b) in the previous lemma is satisfied.  
\end{proof}

\begin{rmk}
    In the above theorem, the first case corresponds to the case when $dX^* = 0$, which has already been discussed.  As for the second case, for dimension three, this is the case where we have a Sasakian structure.  Furthermore, it is known (as noted in \cite{Boyer06} for example) that the transverse space $B$ is in fact K\"ahler in this case.  
\end{rmk}

\noindent Corollary \ref{cor:1} now follows from Theorem \ref{thm:2}.

\begin{rmk}
    In spite of Corollary \ref{cor:1}, if we assume $B$ is not Einstein, it is still possible to construct quasi-Einstein metrics as principle $S^1$ bundles over $B$ even if $B$ is odd dimensional.  As an example, consider the three dimensional Heisenberg group, $H_3(\mathbb{R})$ with a left-invariant metric, $\check{g}$. 
    Now consider a 4-dimensional compact Lie group $G$, whose Lie algebra is given by $\mathfrak{g}= \mathfrak{h}_3 \oplus_\phi \mathbb{R}$ where $\mathfrak{h}_3$ is the Lie algebra of the Heisenberg group and $\phi = \mathrm{ad}_{Z}|\mathfrak{h}_3$ and $Z \in T_eH_{3}(\mathbb{R})$ is a unit vector.  Valiyakath in Theorem 4.14 of \cite{Valiyakath25} shows that there exists a left-invariant quasi-Einstein metric on $G$, where the vector field $X$ is tangent to $\mathfrak{h}_3$.  

    Now consider a metric on $G$ such that the projection map $\pi: G \to B$ is a Riemannian submersion with totally geodesic $S^1$ fibers.  Here, $B$ is the space that results when we quotient out by the action the $S^1$ induced by the vector field $X$.  This gives an example of an even dimensional quasi-Einstein metric of constant scalar curvature constructed as a circle bundle over a compact space where $X$ is tangent to the fibers of the submersion. 
\end{rmk}

\section{The Canonical Variation}

Let $\pi: (M,g) \to B \cong M/S^1$ be a Riemannian submersion with totally geodesic $S^1$ fibers as outlined in the previous section, where $(M,g)$ has constant scalar curvature.  In this section, we are interested in scaling the metric on the vertical distribution $\mathcal{V}$ by a constant factor $t$, while leaving the metric on $\mathcal{H}$ unchanged.  This is called the \textit{canonical variation} of $\pi$, and a major focus of this section will be the finding if and when there exist multiple quasi-Einstein metrics in the canonical variation.  

\begin{defn}
    Let $\pi:(M, g) \to B$ be a Riemannian submersion.  The \textbf{canonical variation of g with respect to $\pi$}, denoted by $g_t$ is defined as follows for $t>0$:
        \begin{enumerate}
            \item[(a)] $g_t(u, v) = g(u, v)$ for $u,v \in \mathcal{H}$.
            \item[(b)] $g_t(y, v) = 0$ for $u \in \mathcal{H}$ and $y \in \mathcal{V}$.
            \item[(c)] $g_t(y, z) = tg(y,z)$ for $y,z \in \mathcal{V}$.
        \end{enumerate}
    Observe that when $t=1$, $g_t = g$.
\end{defn}

\begin{ex}
    Let $\pi: S^3 \to S^2$ be the Hopf fibration, where we equip $S^3$ with the standard round metric, which we call $g$.  The canonical variation of $\pi$, $(M,g_t)$, are known as \textbf{Berger Spheres}.  These metrics define a one parameter family of left-invariant homogeneous metrics on $S^3$ (equivalently $SU(2)$).  More generally, the Hopf fibration can be generalized as $S^{2n+1}$ fibers over $\mathbb{C}P^n$, where the standard fibration is the $n=1$ case, so we can generalize to a map $\pi: S^{2n+1} \to \mathbb{C}P^n$.  In Example 3.5 of \cite{Ghosh19}, Ghosh shows that a $D$-homotheity of $S^{2n+1}$ here is $\eta$-Einstein and admits a quasi-Einstein metric.
\end{ex}

\noindent In general, one can describe the Ricci curvatures $\mathrm{Ric}_t$ in terms of the tensor $A$ associated to the submersion $\pi$ and the Ricci curvature on the base.  We state these formulae below, which follows from Proposition 9.70 of \cite{Besse87} by simply taking the fibers of $\pi$ to be $S^1$.

\begin{prop}[\cite{Besse87}]
    Let $\pi: (M, g) \to B$ a Riemannian submersion with totally geodesic $S^1$ fibers.  Then the Ricci curvatures of the canonical variation $\mathrm{Ric}_t$ are given by:
        \begin{itemize}
            \item[(a)] $\mathrm{Ric}_t(U, V) = t^2g(AU, AV)$ for $U, V \in \mathcal{V}$.
            \item[(b)] $\mathrm{Ric}_t(Y, Z) = \widecheck{\mathrm{Ric}}(Y, Z) - 2tg(A_Y, A_Z)$ for $Y,Z \in \mathcal{H}$.
            \item[(c)] $\mathrm{Ric}_t(Y, U) = tg((\check{\delta}A)(Y, U))$ where $\check{\delta}$ is the codifferential operator on forms on $B$.
        \end{itemize}
\end{prop}

\noindent Combining the above formulae with the quasi-Einstein condition, we obtain the following necessary and sufficient conditions for the canonical variation $(M, g_t, X)$ to be quasi-Einstein for a given $t>0$.

\begin{lem} \label{lem:eq}
    Let $\pi: (M, g) \to B$ a Riemannian submersion with totally geodesic $S^1$ fibers where $(M,g)$ has constant scalar curvature.  Furthermore, suppose $\{X_t\}$ is a family of smooth vector fields on $M$ tangent to the fibers of $\pi$.  Then the one parameter family of triples $(M,g_t, X_t)$ in the canonical variation on $g$ with respect to $\pi$ is quasi-Einstein if and only if the following hold for all $t>0$:
        \begin{itemize}
            \item[(a)] $\mathcal{H}$ is a Yang-Mills connection (i.e. $\check{\delta}A_t=0$).
            \item[(b)] $\widecheck{\mathrm{Ric}}(Y, Z) - 2tg(A_Y, A_Z) = (t|A|^2-\frac{tf^2(t)}{m})g(Y,Z)$ for all horizontal vectors $Y$ and $Z$.
        \end{itemize}
    where $X_t = f(t)U$, and $U$ is a unit vector field (with respect to the original metric $g$ corresponding to $t=1$) tangent to the fibers of $\pi$.
\end{lem}

\begin{proof}
    To begin, we suppose that there exists a one parameter family of quasi-Einstein solutions $(M,g_t, X_t)$ in the canonical variation of $g$ with respect to $\pi$.  Since we assume the family $\{X_t\}$ is tangent to the fibers for $\pi$ for all $t$, we can write $X_t = f(t)U$ for some unknown smooth function $f(t)$.  Thus, we require
        \begin{align} \label{qe-t}
                \mathrm{Ric}_t =\lambda_tg_t + \frac{X_t^* \otimes X_t^*}{m} 
        \end{align} 
    where $t$ subscripts denote with respect to $g_t$ and where we are once again using the fact that $X_t$ must be a Killing field of constant length.  By Proposition 4.3(c), a Yang-Mills connection ($\check{\delta}A_t = 0$) is equivalent to the condition $\mathrm{Ric}_t(U, Z) = 0$ for a vertical vector $U$ and a horizontal vector $Z$.  This condition is satisfied by plunging in $U$ and $Z$ into (\ref{qe-t}), so condition (a) is satisfied.  Using Proposition 4.3(a) in conjunction with (\ref{qe-t}), and plugging in unit length vertical vectors, we obtain $\lambda_t = |A_t|^2-\frac{|X_t|_t^2}{m}$, and so
        \[ \mathrm{Ric}_t(Y,Z) = \left (t|A|^2-\frac{tf^2(t)}{m} \right )g(Y,Z)\]
    for horizontal vectors $Y$ and $Z$.
    Here, we further use the facts that $|A_t|^2 = t|A|^2$ and $|X_t|_t^2 = |f(t)U|_t^2= tf(t)^2|U|^2 = tf(t)^2$.  Norms without a subscript are assumed to be with respect the $t=1$ metric.  Next, using Proposition 4.3(b) to replace the left-hand side of the above, we obtain
        \[\widecheck{\mathrm{Ric}}(Y, Z) - 2tg(A_Y, A_Z) = \left (t|A|^2-\frac{tf^2(t)}{m} \right ) g(Y,Z).\]
     This proves condition (b). \\  

     Conversely, assume conditions (a) and (b) are satisfied for all $t>0$.  Condition (b) implies that (\ref{qe-t}) is satisfied for horizontal vectors, and condition (a) implies (\ref{qe-t}) is satisfied when one entry is a vertical vector and the other is a horizontal vector.  Plugging (unit length) vertical vectors into (\ref{qe-t}) and using Proposition 4.3(a), we obtain $\lambda_t = |A_t|^2-\frac{|X_t|_t^2}{m}$, which is implied by condition (b).  This shows (\ref{qe-t}) is indeed satisfied for any choice of vectors, and so the claim is proved.
\end{proof}

\noindent Note that this proposition reduces to Proposition 3.8 in the case where $t=1$ and the fibers have length $2\pi$. \\

Given this notion of canonical variation, a natural question to ask is: If $(M, g, X)$ is quasi-Einstein, under what conditions does there exist another metric $(M,g_t, X_t)$ in the canonical variation?  Furthermore, given an Einstein metric $(M,g,X)$, is there ever another Einstein metric in the canonical variation?  We answer these questions below.  We begin by proving the following lemma.

\begin{lem} \label{lem:c_t}
     Let $\pi: (M,g) \to B$ be a Riemannian submersion with totally geodesic $S^1$ fibers, where $(M,g)$ has constant scalar curvature.  Further suppose that $(M,g,X)$ is quasi-Einstein where $X$ is tangent to the fibers of $\pi$.  If $(M, g_{t_0}, X_{t_0})$ is another quasi-Einstein metric in the canonical variation of $\pi$, then $X = c_{t_0}U$, where
        \[c_{t_0} = \sqrt{\frac{|X|^2}{t_0} + \frac{m(t_0-1)(n+1)|A|^2}{t_0(n-1)}}.\]
    \noindent Here, $|X|$ is with respect to the $t=1$ metric, $U$ is a unit vector field with respect to the $t=1$ metric tangent to the fibers of $\pi$, and $n$ is the dimension of $M$.
\end{lem}

\begin{proof}
    Since $(M,g,X)$ is quasi-Einstein, we have that
        \begin{align} \label{qe-o00}
            \mathrm{\widecheck{Ric}}(Y,Z) - 2g(A_Y, A_Z) = \left ( |A|^2 - \frac{|X|^2}{m} \right )g(Y,Z) 
        \end{align}

    \noindent for arbitrary horizontal vectors $Y$ and $Z$.  
    Now, by the previous lemma, if $(M, g_{t_0}, X_{t_0})$ is quasi-Einstein for some number $t \neq 1$, then 
        \begin{align} \label{eq:3}
            \mathrm{\widecheck{Ric}}(Y,Z) - 2t_0g(A_Y, A_Z) = \left ( t_0|A|^2-\frac{tc_{t_0}^2}{m} \right )g(Y,Z).
        \end{align}

    \noindent where $c_{t_0}$ is the constant satisfying $X_{t_0} = c_{t_0}U$.  Subtracting (\ref{eq:3}) from (\ref{qe-o00}), one obtains 
    
        \begin{align} \label{qe-o2}
            2(1-t_0)g(A_Y, A_Z) = \left ( (1-t_0)|A|^2 + \frac{tc_{t_0}^2-|X|^2}{m} \right ) g(Y,Z)
        \end{align}

    \noindent (where again $|X|$ is the norm of $X$ in the $t=1$ metric).  Taking the trace of (\ref{qe-o2}) over an orthonormal basis $\{X_i\}$ of the horizontal distribution, one obtains:

        \begin{align} \label{qe-o3}
            2(t_0-1)|A|^2 = \left ( (1-t_0)|A|^2 + \frac{tc_{t_0}^2-|X|^2}{m} \right ) (n-1)
        \end{align}

    \noindent where $n$ is the dimension of $B$.  Rearranging (\ref{qe-o3}) yields the following expression for $c_{t_0}$:

        \[c_{t_0} = \sqrt{\frac{|X|^2}{t_0} + \frac{m(t_0-1)(n+1)|A|^2}{t_0(n-1)}} \]

    \noindent as claimed.
\end{proof}

We are now ready to show there are two quasi-Einstein metics in the canonical variation if and only if the base is Einstein.

\begin{thm} \label{qe-b}
    Let $\pi: (M,g) \to B$ be a Riemannian submersion with totally geodesic $S^1$ fibers where $(M,g)$ has constant scalar curvature, and suppose $(M,g,X)$ is quasi-Einstein where $X$ is tangent to the fibers of $\pi$.  Then there exists a distinct quasi-Einstein metric $(M,g_{t_0},X_{t_0})$ (where $t_0 \neq 1$) in the canonical variation of $g$ with $X_{t_0}$ tangent to the fibers of $\pi$ if and only if $B$ is Einstein.  In this case, we have that $X_{t_0} = c_{t_0}U$, where $c_{t_0}$ is the constant defined in the previous lemma, and $U$ is a unit vector field with respect to the $t=1$ metric tangent to the fibers of $\pi$.
\end{thm}

\begin{proof}
    First suppose $B$ is Einstein.  Then $\mathrm{\widecheck{Ric}} = \check{\lambda}g$ for some constant $\check{\lambda}$.  We have that $\check{\delta}(A_t)= t\check{\delta}(A) = 0$ since $(M,g,X)$ is quasi-Einstein by assumption, and so condition (a) of Lemma 4.4 is satisfied.  Since $(M,g,X)$ is quasi-Einstein, we have that
        \begin{align} \label{qe-o0}
            \mathrm{\widecheck{Ric}}(Y,Z) - 2g(A_Y, A_Z) = \left ( |A|^2 - \frac{|X|^2}{m} \right )g(Y,Z) 
        \end{align}

    \noindent for arbitrary horizontal vectors $Y$ and $Z$.  
    Now, by the previous lemma, if $(M, g_{t_0}, X_{t_0})$ is quasi-Einstein for some number $t \neq 1$, then $X_{t_0} = c_{t_0}U$ where the constant $c_{t_0}$ is given in the statement of the previous lemma.  By Lemma \ref{lem:eq}, it now suffices to show that 
        \begin{align} \label{eq:2}
            \mathrm{\widecheck{Ric}}(Y,Z) - 2t_0g(A_Y, A_Z) = \left ( t_0|A|^2-\frac{tc_{t_0}^2}{m} \right )g(Y,Z).
        \end{align}

    \noindent for some positive constant $t_0 \neq 1$.  Using the assumption that $B$ is Einstein and rearranging, we can rewrite equations (\ref{qe-o0}) and (\ref{eq:2}), respectively, as follows:

        \begin{align} \label{qe-o0_1}
            2g(A_Y, A_Z) = \left ( \check{\lambda} - |A|^2 + \frac{|X|^2}{m} \right )g(Y,Z) 
        \end{align}

    \noindent and

        \begin{align} \label{eq:2_1}
            2t_0g(A_Y, A_Z) = \left ( \check{\lambda} - t_0|A|^2 + \frac{t_0c_{t_0}^2}{m} \right )g(Y,Z).
        \end{align}

    Plugging in the expression for $c_0$ found in the previous lemma,(\ref{eq:2_1}) can be rewritten as

        \begin{align} \label{eq:2_2}
            2t_0g(A_Y, A_Z) = \left ( \check{\lambda} - t_0|A|^2 + \frac{|X|^2}{m}+\frac{(t_0-1)(n+1)|A|^2}{n-1} \right )g(Y,Z).
        \end{align}

    We claim that for a choice of $t_0 \neq 1$ such that $c_{t_0}$ is well-defined, (\ref{qe-o0_1}) implies (\ref{eq:2_1}), which would imply the forward direction of the theorem.  To see why this implication holds, we first multiply (\ref{qe-o0_1}) by our choice of $t_0$ to obtain

    \begin{align} \label{qe-o0_2}
        2t_0g(A_Y, A_Z) = \left ( \check{\lambda}t_0 - t_0|A|^2 + \frac{t_0|X|^2}{m} \right )g(Y,Z).
    \end{align}

    Our goal will be to show that the expression inside parentheses of (\ref{eq:2_1}) and (\ref{qe-o0_1}) are equivalent.  To do this, we will first derive an expression for $\check{\lambda}$ which involves only $|A|$, $|X|$, $n$ (the dimension of $M$) and $m$.  Taking the trace of Proposition 4.3(b) when $t=1$, in combination with the quasi-Einstein condition, we obtain:

    \begin{align}
        \check{\lambda}(n-1)-2|A|^2 = \left ( |A|^2 - \frac{|X|^2}{m}\right ) (n-1).
    \end{align}

    \noindent Solving for $\check{\lambda}$ yields

    \begin{align} \label{eq:lambda}
        \check{\lambda}=\frac{n+1}{n-1}|A|^2-\frac{|X|^2}{m}.
    \end{align}

    Finally, we can substitute our expression for $\check{\lambda}$ into (\ref{qe-o0_1}) and (\ref{eq:2_2}).  Doing so, we see that they are equal.  Seeing that the former implies the latter, this shows that $(M,g_{t_0},X_{t_0})$ is quasi-Einstein as claimed.

    Conversely, suppose that there exists two distinct quasi-Einstein metrics $(M,g,X)$ and $(M,g_{t_0}, X_{t_0})$ in the canonical variation of $\pi$.  Then, we have that

     \[ \widecheck{\mathrm{Ric}}(Y, Z) - 2g(A_Y, A_Z) = \left (|A|^2-\frac{|X|^2}{m} \right )g(Y,Z) \]

    and 

    \[ \widecheck{\mathrm{Ric}}(Y, Z) - 2t_0g(A_Y, A_Z) = \left (t_0|A|^2-\frac{|X|^2}{m} \right )g(Y,Z) \]

     for all horizontal vectors $Y$ and $Z$ and $t>0$.  Now, we can take the two equations and subtract one from the other.  Doing so, we obtain

    \[  - 2(1-t_0)g(A_Y, A_Z) = \left ((1-t_0)|A|^2-\frac{|X|^2-t_0|X_{t_0}|}{m} \right )g(Y,Z). \]

    Hence, $g(A_Y,  A_Z) = kg$ for some constant $k$, since $|A|$ must be constant by Lemma 3.4, and we know from earlier that $|X|$ must be constant as well.  It follows that $\widecheck{\mathrm{Ric}}$ is Einstein, since it is a sum of two tensors proportional to $g$.
\end{proof}

\begin{rmk}
    Note that the argument above shows that if $(M,g_1, X_1)$ and $(M,g_{t_0}, X_{t_0})$ (for $t_0 \neq 1)$ are both quasi-Einstein, then $(M,g_t,X_t)$ is also quasi-Einstein at all $t$ values where the constants $c_t$ are well defined.  To find all such $t$ values, we are simply looking for conditions on $t, m, |X|^2, n$ and $|A|^2$ such that
        \[ \frac{|X|^2}{t} + \frac{m(t-1)(n+1)|A|^2}{t(n-1)} \geq 0  \] or, equivalently, 
    \[ t \geq 1 -\frac{|X|^2(n-1)}{m(n+1)|A|^2}.  \]
    Observe that this condition is always satisfied if $m > 0$ and $t>1$.  If $m<0$, then the condition is a little more restrictive.  However, we can observe it is satisfied if $t > 1$ and $|X|=0$, for instance.
\end{rmk}

\begin{proof}[Proof of Corollary 1.6:]
    Let $(B,\check{g})$ be a compact Einstein manifold, and suppose $p:P \to B$ is a nontrivial principal $S^1$ bundle with $\check{\lambda>0}$ and $m>0$.  Suppose $P$ admits a quasi-Einstein metric.  By the work above, there exist an infinite family of quasi-Einstein metrics $(M,g_t,X_t)$ in the canonical variation.  It suffices to find a $t$ value such that $c_t = 0$.  Setting $c_{t_0} = 0$, we obtain the condition
        \[ t = 1 -\frac{|X|^2(n-1)}{m(n+1)|A|^2}. \]
    Since $\check{\lambda}>0$ by assumption, (4.13) implies 
        \[ \frac{n+1}{n-1}|A|^2 > \frac{|X|^2}{m} \]
    \noindent which in turn implies that there is a $t \in (0,1)$ such that $c_t = 0$.  Hence, there is an Einstein metric in the canonical variation as claimed.

    Furthermore, an expression for $\lambda_t$ can be easily be extracted from Equation (\ref{eq:2_2}) by looking at the right-hand side of the equation after the $\check{\lambda}$ term.  We see that

    \[ \lambda_t = t|A|^2-\frac{|X|^2}{m}-\frac{(t-1)(n+1)|A|^2}{n-1}. \]

    \noindent Observe that $\lambda_t$ is a decreasing linear function.  Furthermore, assuming $\check{\lambda}>0$, we see that $\lambda_t>0$ when $t=1$, so we will have metrics with $\lambda_t$ positive, negative, and zero in the canonical variation as claimed.
\end{proof}

\begin{ex}
        Consider the Hopf fibration $\pi: S^3 \to S^2$ with the round metric of radius 1 equipped on $S^3$ when $t=1$.  We consider $\{U, X, Y \}$ to be an orthonormal frame and $S^3$ to be equipped with the round metric with constant curvature 1 when $t=1$.  Under the Hopf fibration, $S^2$ has radius 1, and $\check{\lambda}=1$.  We compute $|A|^2$ as follows (the covariant derivatives for $S^3 \cong SU(2)$ in this case can be found in \cite{Petersen16}):

    \begin{align}
        |A|^2 &= g(A_X U, A_X U) + g(A_Y U, A_Y U) \\\notag &= g(\mathcal{H}\nabla_{X}U, \mathcal{H}\nabla_{X}U) + g(\mathcal{H}\nabla_{Y}U, \mathcal{H}\nabla_{Y}U) \\\notag&= g(Y, Y) + g(-X, -X) = 1 + 1 = 2.
    \end{align}

    \noindent Plugging this value of $|A|^2$ into the expression found for $X_t$, we obtain $X_t = \pm \sqrt{4m-\frac{4m}{t}}U$.  We verify that $X = 0$ for $t=1$ as expected.  This expression also agrees with what Lim showed in \cite{Lim22}.  
\end{ex}

Observe that it follows from Proposition 2.5 that if $m, \lambda_t < 0$, there are no solutions.  More generally, in \cite{Colling25}, Colling-Dunajski show that there are also no solutions when $\lambda<0$ and $m+n-2 < 0$ in the non-constant scalar curvature case (here, $n$ is the dimension of the manifold).

\section{The Integral Curves of $X$ Must be Closed in Dimension Three}

In Section 2, we classified compact quasi-Einstein metrics of constant scalar curvature in dimension three.  From the classification result, it follows that all such metrics can be realized as circle bundles over a manifold base, and so, the geodesics generated by $X$ must be closed.  In this section, we will provide an alternate proof that the geodesics generated by $X$ must be closed in the three dimensional case.  We begin with a couple definitions:

\begin{defn}
    Let $G$ be a compact Lie group acting on a space $M$.  The action of $G$ on $M$ is called \textbf{effective} if $g \cdot x = x$ implies $g=e$ for all points $x \in M$.
\end{defn}

\begin{defn}
    Let $G$ be a compact Lie group acting on a space $M$.  The action of $G$ on $M$ is called $\textbf{almost free}$ if for all points $x \in M$, the isotropy group at $x$, $G_x = \{g \in G | g \cdot x = x\}$ is finite.
\end{defn}

Now, we will state and prove a useful Lemma.

\begin{lem} \label{lem:1}
    Let $X$ be a constant length Killing field on a compact Riemannian manifold $(M, g)$.  Then either:
        \begin{enumerate}
            \item[1)] There exists a Riemannian submersion $\pi: M \to B \cong M/S^1$ with totally geodesic $S^1$ fibers, $B$ is an orbifold, and $X$ is tangent to the fibers of $\pi$ or
            \item[2)] There exists an isometric $\mathbb{T}^2$ action on $M$ (where $\mathbb{T}^2$ is $S^1 \times S^1$).
        \end{enumerate}
\end{lem}

\begin{proof}
    First, we assume that the integral curves generated by $X$ are closed.  Let $\gamma_x$ be the integral curve generated by $X$ passing through some initial point $x \in M$.  Note that $\gamma_x$ is one dimensional since $X$ is constant length, meaning that it vanishes nowhere (assuming $X$ isn't identically zero).  Since $\gamma_x$ is also closed, it must be topologically $S^1$.  Furthermore, we observe that $0 = D_Y|X|^2 = g(\nabla_Y X, X) = -g(\nabla_X X, Y)$ for an arbitrary smooth vector field $Y$, showing that $\gamma_x$ is totally geodesic.  Here, the first equality uses the assumption $X$ has constant norm, and the third uses the fact that $X$ is Killing, and hence the covariant derivative is antisymmetric.  By Theorems 14 and 15 of \cite{Berestovskii2024}, we also have that there exists a smooth, effective, isometric, and almost free $S^1$ action on $M$.  We can then define $\pi: M \to B \cong M/S^1$ as the canonical projection with the quotient metric.  Checking that $d\pi$ is an isometry, we see that $\pi$ is indeed a Riemannian submersion as desired.  \\

   \noindent Next, we assume that the integral curves $\gamma_x$ generated by $X$ are not closed.  Let $G$ be the one parameter family of isometries of $M$ generated by $X$.  We have that $G \subset \mathrm{Iso}(M,g)$, and since $\mathrm{Iso}(M,g)$ is a closed subgroup, it is also true that $\bar{G} \subset \mathrm{Iso}(M,g)$, where $\bar{G}$ denotes the closure of $G$ in $M$.  Since $\mathrm{Iso}(M,g)$ is a Lie group and since $\bar{G}$ is a closed subgroup of $\mathrm{Iso}(M, g)$, Cartan's theorem tells us that $\bar{G}$ is in fact a closed Lie subgroup of $\mathrm{Iso}(M,g)$.  Furthermore, $\bar{G}$ is an \textit{abelian} Lie subgroup.  To see why, observe that $G$ is abelian since it is a one dimensional Lie group.  Now consider arbitrary points $a, b \in \bar{G}$.  Since $\bar{G}$ is the closure of $G$, we can choose sequences $a_i, b_i \in G$ where $b_i$ converges to $b$ and $a_i$ converges to $a$.  Now consider the sequence $g_i = a_ib_ia_i^{-1}b_{i}^{-1} \in G$.  Since $G$ is abelian, clearly $g_i = e$ (the identity element) for all $i$, and so $g_i$ converges to $e$.  However, by continuity, we also know that $g_i$ converges to $aba^{-1}b^{-1}$, and so $\bar{G}$ is abelian as desired.  \\

   \noindent Since $\bar{G}$ is a closed Lie subgroup in a closed manifold $M$, we further know that $\bar{G}$ is also compact.  It is also clear that $\bar{G}$ is connected since $G$ is connected.  Hence, $\bar{G}$ is a connected, compact, abelian Lie group, and hence must be diffeomorphic to a torus.  Note that if $\bar{G} \cong S^1$, that would contradict the fact that the geodesics $\gamma_x$ are not closed.  Hence, $\mathbb{T}^2 \subset \bar{G}$, and thus $M$ has a 2-torus acting on it by isometries as claimed. 

\end{proof}

Suppose that the geodesics generated by the vector field $X$ are not closed.  As shown in the above lemma, we have an isometric $\mathbb{T}^2$ action on $M$.  Another way to express this fact is that there exist two commuting Killing fields on $M$ whose integral curves are closed, which we shall call $K_1$ and $K_2$.  Without loss of generality, we can also assume that our Killing field $X$ is a linear combination of $K_1$ and $K_2$, that is, $X = aK_1 + bK_2$ for non-zero constants $a, b$.  By our previous assumptions, $|X|$ is constant.  We will show that given these assumptions, $(M, g, X)$ cannot be quasi-Einstein for nonzero $X$. We begin by introducing a few definitions and lemmas which will come in handy later.

\begin{defn}
    Let $G \subset \mathrm{Iso}(M,g)$ be a closed Lie group of isometries of a closed Riemannian manifold $M$.  $M$ is said to have \textbf{cohomogeneity one} if the orbit space $\Omega = M/G$ is one dimensional.
\end{defn}

Note that in general, $\Omega$ may have boundary points even if $M$ does not.  In fact, $\Omega$ may have singular points, and in general, is only an orbifold and not a manifold.  In the cohomogeneity one case, $\Omega$ must be either a closed interval or $S^1$.  In the former case, we want to distinguish points corresponding to boundary points of $\Omega$ and those corresponding to interior points.

\begin{defn}
    Suppose $(M,g)$ is a closed cohomogeneity one manifold under the action of a group $G \subset \mathrm{Iso}(M,g)$.  We say $x \in M$ is a $\textbf{regular point}$ if the image of $x$ in $M/G$ is a interior point, and an \textbf{exceptional point} otherwise.
\end{defn}

Equipped with the above definition, we state the following proposition:

\begin{prop}
    Let $x$ be a regular point of cohomogeneity one 3-manifold $M$ with respect to a group $G \subset \mathrm{Iso}(M,g)$.  Then
        \begin{itemize}
            \item[(a)] There exists a neighborhood $U$ of $x$ such that the metric on $U$ can be written in the form $g = dr^2 + f^2(r)d\theta_1^2+h^2(r)d\theta_2^2$.
            \item[(b)] For the neighborhood $U$ and the metric $g$ in part (a) above, the Ricci tensor is diagonal.  
        \end{itemize}
\end{prop}

\begin{proof}
    For part (a), we refer to Proposition 1.4 (e) in \cite{Podesta94}.  For part (b), observe that the metric in part (a) only depends on one of the coordinates.  By an argument in Section 4.2.4 of \cite{Petersen16}, this implies the Ricci tensor is diagonal with respect to these coordinates.
\end{proof}

The above proposition is useful to prove the following rigidity result, showing that quasi-Einstein solutions on closed manifolds of constant scalar curvature where the integral curves of $X$ are not closed are not possible in dimension three.

\begin{thm}
    Let $X$ be a smooth vector field on a closed manifold $(M,g)$ of constant scalar curvature with $\mathrm{dim}(M)= 3$.  If $X$ is not identically zero and if the integral curves generated by $X$ are not closed, then the triple $(M, g, X)$ cannot be quasi-Einstein.
\end{thm}

\begin{proof}
    Suppose $(M, g, X)$ is a quasi-Einstein metric, where $(M, g)$ and $X$ satisfy the above assumptions.  As shown in \cite{Ghosh20} and \cite{Bahuaud24}, $X$ must be Killing and $|X|$ must be constant on $M$.  The quasi Einstein equations therefore yield
        \begin{align}
            \mathrm{Ric}(X, X) = \lambda |X|^2 + \frac{|X|^4}{m}.
        \end{align}
        
    or, equivalently, 

        \begin{align}
            \mathrm{Ric}(X) = \lambda X + \frac{|X|^2}{m}X = \left (\lambda+\frac{|X|^2}{m} \right )X
        \end{align}

    \noindent where we define $\mathrm{Ric(X)}$ to be the unique vector such that $\mathrm{Ric}(X, X) = g(\mathrm{Ric}(X), X)$. \\
        
    This shows that $X$ is an eigenvector for the Ricci tensor.  Since the integral curves of the Killing field $X$ are not closed, there must be an isometric $\mathbb{T}^2$ action on $M$ (whose orbits contain integral curves of $X$) by Lemma \ref{lem:1}.  Now, let $x \in M$ be a regular point in $M$ with respect to this $\mathbb{T}^2$ action.  Note that $M$ is cohomogeneity one with respect to this action.  By the previous proposition, in a neighborhood $U$ around $x$, the metric can we written as $g = dr^2+f^2(r)d\theta_1^2+h^2(r)d\theta_2^2$ where the Killing vector fields $\frac{\partial}{\partial\theta_1}, \frac{\partial}{\partial\theta_2}$ generate the $\mathbb{T}^2$ action.  Also by the previous proposition, the Ricci tensor is diagonal in $U$, and so we can assume that $\frac{\partial}{\partial\theta_1}$ and $\frac{\partial}{\partial\theta_2}$ are eigenvectors for the Ricci tensor.  Let $\lambda_1$ and $\lambda_2$ be the corresponding eigenvalues, respectively.  \\
    
    Recall that since $X$ is Killing, and since the integral curves of $X$ are contained in the $\mathbb{T}^2$ action, we can write $X = a \frac{\partial}{\partial\theta_1} + b \frac{\partial}{\partial\theta_2}$ for constants $a, b$.  We have that 
    
    \[\mathrm{Ric}(X)= \mathrm{Ric} \left (a \frac{\partial}{\partial\theta_1}+b \frac{\partial}{\partial\theta_2} \right ) = a\lambda_1\frac{\partial}{\partial\theta_1} + b\lambda_2 \frac{\partial}{\partial\theta_2} \]

    \noindent Continuing this calculation, and recalling $X$ is an also eigenvalue of the Ricci tensor, we further have that 

    \begin{align*}
        \mathrm{Ric}(X) &= a\lambda_1g\frac{\partial}{\partial\theta_1} + b\lambda_2 \frac{\partial}{\partial\theta_2} \\ & = \lambda_3X = \lambda_3 \left (a \frac{\partial}{\partial\theta_1} + b \frac{\partial}{\partial\theta_2}\right ) \\&= a\lambda_3 \frac{\partial}{\partial\theta_1} + b\lambda_3g\frac{\partial}{\partial\theta_2}.
    \end{align*}

    \noindent implying that

    \[ a(\lambda_1-\lambda_3)\frac{\partial}{\partial\theta_1}+ b(\lambda_2 - \lambda_3)\frac{\partial}{\partial\theta_2} = 0. \]

    \noindent By linear independence of $\frac{\partial}{\partial \theta_1}$ and $\frac{\partial}{\partial \theta_2}$, we can conclude that $\lambda_1 = \lambda_2 = \lambda_3$, and thus, the multiplicity of the eigenvalue associated to $X$ has multiplicity two (or greater).  However, by the quasi-Einstein equation, the eigenspace of the eigenvalue associated to $X$ has dimension one, which is a contradiction.  Therefore, we conclude that $(M,g,X)$ cannot be quasi-Einstein in this case, as desired.
    
\end{proof}

\section*{Acknowledgments}

The author of this paper would like to thank his PhD advisor, William Wylie, for his helpful guidance, comments, and insights during the process of writing this paper.

\bibliographystyle{alpha}
\bibliography{bibfile}

\vspace{3mm}

\noindent \textsc{215 Carnegie Building, Dept. of Math,
   Syracuse University,
 Syracuse, NY, 13244.}
\textit{E-mail: }\href{mailto:ecochr01@syr.edu}{ecochr01@syr.edu}

\end{document}